\documentclass[a4paper,11pt]{amsart}
\usepackage[margin=2.7cm]{geometry}

\usepackage{amsfonts,amssymb,amsmath}
\usepackage{mathrsfs} 
\usepackage{microtype}
\usepackage{mlmodern}
\usepackage[T1]{fontenc}

\numberwithin{equation}{section}

\newtheorem{thm}{Theorem}[section]
    \newtheorem{Proposition}{Proposition}[section]
    \newtheorem{cor}{Corollary}[section]
    \newtheorem{lem}{Lemma}[section]
    \newtheorem{rmk}{Remark}[section]
    
    \newtheorem{exmp}{Example}[section]

\allowdisplaybreaks

\newcommand{\stirling}{\genfrac{\{}{\}}{0pt}{}}

\begin{document}

\title[Extended COM-Poisson distribution]{On count data models based on Bernstein functions or their inverses}

    \author[G. D'Onofrio]{Giuseppe D'Onofrio$^{\circ}$}
        \address{$^{\circ}$ Dipartimento di Scienze Matematiche, Politecnico di Torino, Corso Duca degli Abruzzi 24, 10129 Torino, Italy}
        \email{giuseppe.donofrio@polito.it}

    \author[F. Polito]{Federico Polito$^{\dagger}$}
        \address{$^{\dagger}$ Dipartimento di Matematica \lq\lq G. Peano\rq\rq, Università degli Studi di Torino, Via Carlo Alberto 10, 10123 Torino, Italy}
        \email{federico.polito@unito.it}

    \author[\v Z. Tomovski]{\v Zivorad Tomovski$^{\ddagger}$$^{\star}$}

        \address{$^{\ddagger}$ Institute for Research and Applications of Fuzzy Modeling, University of Ostrava, 30. dubna 22 702 00 Ostrava, Czech Republic}
        \email{zhivorad.tomovski@osu.cz}
        
        \address{$^{\star}$ Department of Mathematical Analysis and Applications of Mathematics, Faculty of Science Palacký University Olomouc, 17. listopadu 12, 771 46 Olomouc, Czech Republic}
        \email{zhivorad.tomovski@upol.cz}

    \begin{abstract}
        We present a class of positive discrete random variables extending the Conway--Maxwell-Poisson distribution. This class emerges in a natural way from an application in queueing theory and contains distributions exhibiting quite different features. Some of these distributions are characterized by the presence of Bernstein and inverse Bernstein functions. As a byproduct, we give some results on these inverses for which the existing literature is limited. Moreover, we investigate dispersion properties for these count data models, giving necessary and/or sufficient conditions to obtain both over and underdispersion.
        We also provide neat expressions for the factorial moments of any order. This furnishes us with a compact form also in the case of the Conway--Maxwell-Poisson.     
    \end{abstract}

    \medskip
    
    \noindent

    \maketitle

    \section{\label{sec:level1}Introduction}

        Starting from the seminal paper by Conway and Maxwell \cite{conway1962queuing}, the interest in the COM-Poisson distribution (also known as CMP distribution) as a model of count data, has  grown continuously in different fields such as statistics \cite{sellers2010flexible, wu}, social and natural sciences \cite{beisemann2022flexible,puig2023some}, economics \cite{sellers2012poisson, shmueli2005useful} (see also the recent monograph by Sellers \cite{sellers2023conway}). 
        By looking at the probability mass function of the COM-Poisson random variable,
        \begin{eqnarray} 
            f_n(\zeta)=\frac{\zeta^n}{n!^\delta}\frac{1}{C_\delta(\zeta)}, \qquad C_\delta(\zeta)=\sum_{n=0}^{\infty}\frac{\zeta^n}{n!^\delta}. \qquad n \in \mathbb{N} = \{ 0,1,\dots\}, \: \delta,\zeta >0,
        \end{eqnarray}
        we note that its structure strongly resembles that of its celebrated special case, the Poisson random variable ($\delta=1$). Despite this similarity, the two distributions differ on a fundamental aspect: the COM-Poisson can exhibit a variance to mean ratio smaller, equal, or larger than unity.
        A long list of variants of the Poisson and COM-Poisson distributions have been proposed during the last decades to address specific modeling needs. The related literature is  extensive and therefore we
        refer to the following recent articles: \cite{MR2535014, cahoy2021flexible, cahoy2013renewal,dicr2015fractional,garra2018note,MR3543682,meerschaert2011fractional,michelitsch2020generalized,pogany2016probability,siri2012parametric}.
        However, by inspecting the very form of the probability mass function of the COM-Poisson random variable a natural extension emerges.
        The power function which distinguishes the COM-Poisson by the Poisson is, in fact, the most prominent example of a Bernstein function or of an inverse Bernstein function, depending on whether the exponent is respectively smaller or larger than unity.
        Led by this intuition, we consider a class of random variables, in which the role of the power function is played by $\phi$, chosen in the space of Bernstein functions or that of their inverses:
        \begin{eqnarray}  \label{pn_COMintro}
            P_0^\phi(\zeta)=\frac{1}{Z(\zeta,\phi)}, \quad P_n^{\phi}(\zeta)=\frac{\zeta^n}{\prod_{k=1}^n \phi(k)}\frac{1}{Z(\zeta,\phi)}, \quad n \in \mathbb{N}^* =\{1,2,\dots\}, \: \zeta>0,
        \end{eqnarray}
        where
        \begin{equation}\label{thermosintro}
            Z(\zeta,\phi)=1+\sum_{n=1}^{\infty}\frac{\zeta^n}{\prod_{k=1}^n \phi(k)}.  
        \end{equation}
        This class of random variables turns out to be very rich and includes the COM-Poisson and the Poisson random variables as well as several other variants of the COM-Poisson that appeared in the literature.
        Our intuition is confirmed by a model of interest in queueing theory with state-dependent service times. We outline it in Section \ref{sec:level2} in the spirit of \cite{conway1962queuing}, in which \eqref{pn_COMintro} arises as a stationary distribution.

        Bernstein functions are characterized by the regular behavior of their derivatives, have a useful integral representation and are often studied thanks to their connection to L\'evy processes (for this reason they also go under the name of Laplace exponents), see \cite{schilling2012bernstein} for a full discussion.
        To investigate the properties of 
        \eqref{pn_COMintro} as a function of $\phi$, we make use of the Bernstein-gamma function \cite{patie2018bernstein}. The latter is the solution $W_\phi$ to the functional equation
        $W_{\phi}(z+1)=\phi(z) W_{\phi}(z)$, $W_\phi(1)=1$, and it generalizes the classical gamma function.
        The Bernstein-gamma function has been used, for example,  to  generalize the Mittag--Leffler function  \cite{patie2021self}
       and  to study the properties of the Wright function      \cite{bartholme2021turan}.
       Here, we define and outline some properties of the compositional inverse of Bernstein  functions. These inverses appear, for instance, in \cite{jedidi}. However, to the best of our knowledge, an extended treatment on inverses is missing in the literature. On this work, we contribute by outlining some of their properties. Moreover, we  obtain an analog of the Bernstein-gamma function  for eventually log-convex functions in the spirit of \cite{webster}. Among the strengths of these functions, in relation to count data models, we mention their convexity that is ultimately related to the underdispersion property of \eqref{pn_COMintro}.
   
    For a random variable, the property of underdispersion refers to the variance to mean ratio being less than unity. If, otherwise, it is greater than unity, we speak of overdispersion. In the literature on count data models these properties are always put in contrast to the undesired mean-equal-variance property of the Poisson model. Real-world count data quite often exhibit overdispersion (e.g.\ heavy-tails), but underdispersion is also possible (e.g.\ zero-inflation). Overdispersed models arises with relatively simple constructions (see \cite{hinde1998overdispersion} for a discussion), while models capable of dealing with underdispersed data are less frequent in the literature (see e.g.\ \cite{ consul1973generalization,MR2200069,huang2023arbitrarily,kemp1968wide,kemp1974family,puig2023some,pujol2014new,sellers2010flexible,zhu2012modeling}).
    We provide sufficient conditions for the underdispersion (overdispersion) according to the properties of $\phi$. 
    In particular, the class of models \eqref{pn_COMintro} can be either overdispersed or underdispersed, depending on whether $\phi$ is chosen in the set of Bernstein functions or of their inverses.

    The paper is organized as follows: Section \ref{sec:level2} describes the queueing model in which \eqref{pn_COMintro} arises as a stationary distribution. In Section \ref{sec:bern}, Bernstein functions are recalled, together with some of their properties. The set of inverse Bernstein functions is presented in the same section, while some of their properties are analyzed in the following Section \ref{fff}.
    Section \ref{sec:level4b} deals with fundamental properties of the introduced model, while Section \ref{disssp} is specifically devoted to the analysis of its dispersion properties. Lastly, in Section \ref{extension}, an extension of the model is explored, obtaining, among other results, a compact expression for the factorial moments.

    \section{\label{sec:level2}The model}

        Following \cite{conway1962queuing} we consider a model of the service in a queueing system whose rate depends on the state of the system.
        In particular, we consider the mean service rate $\mu_n=\phi(n)\mu$, where $n$ is the number of units in the system  and  $\phi:(0,\Lambda) \to (0,\infty)$  
        where $\Lambda$ can be infinity. Further, $1/(\mu\phi(1))$ is the service time if only one unit is in the system.
        Let now define $N(t)$, $t \ge 0$, as the random number of units in the system at time $t$ and $P_n(t) =\mathbb{P}(N (t) = n)$, $n \in \mathbb{N}_\Lambda=\mathbb{N}\cap [0,\Lambda]$.
        We furthermore assume that the inter-arrival times are independent and exponentially distributed with mean $\lambda$. The function $ \phi$ is chosen such that $N$ admits a stationary distribution, that is equation (4) of \cite{abouee2016state} is satisfied:
\begin{equation}\label{cond_conv}
\sum_{i=1}^{\lfloor\Lambda\rfloor}  \frac{\rho^i}{\prod_{j=1}^i \phi(j)} < \infty,
\end{equation}
where $\rho=\lambda/\mu$.

By using the notation of \cite{conway1962queuing} the system of differential difference equations is, for every $n\in\mathbb{N}_\Lambda$:
\begin{eqnarray}
P_0(t+\Delta)&=&(1-\lambda \Delta)P_0(t)+\phi(1)\mu\Delta P_1(t), \nonumber \\
P_n(t+\Delta)&=&(1-\lambda \Delta -\phi(n)\mu \Delta)P_n(t)+\lambda \Delta P_{n-1}(t)+\phi(n+1)\mu \Delta P_{n+1}(t).  \nonumber 
\end{eqnarray}
Letting $\Delta \rightarrow 0$
\begin{eqnarray}
P_0^\prime(t)&=&-\lambda P_0(t)+\phi(1)\mu P_1(t), \nonumber \\
P_n^\prime(t)&=&-(\lambda +\phi(n)\mu) P_n(t)+\lambda P_{n-1}(t)+\phi(n+1)\mu P_{n+1}(t). \nonumber
\end{eqnarray}
Since $N(t)$ admits a stationary distribution,
 we get
\begin{eqnarray}\label{extendedPO}
\frac{\rho}{\phi(1)} P_0&=&P_1, \nonumber \\
(\rho + \phi(n))P_n&=&\rho P_{n-1}+\phi(n+1)P_{n+1}, \qquad n \in \mathbb{N}_\Lambda^* = \mathbb{N}_\Lambda \backslash \{0\},\nonumber 
\end{eqnarray}
and hence
\begin{eqnarray} \label{pn}
P_n=\frac{\rho^n}{\prod_{k=1}^n \phi(k)}P_0. 
\end{eqnarray}

Then, from \eqref{pn}, using the normalization condition, we obtain
\begin{eqnarray} \label{pn_COM}
P_0^\phi(\rho)=\frac{1}{Z(\rho,\phi)}, \qquad P_n^{\phi}(\rho)=\frac{\rho^n}{\prod_{k=1}^n \phi(k)}\frac{1}{Z(\rho,\phi)}, \qquad n \in \mathbb{N}_\Lambda^*,
\end{eqnarray}
where
\begin{equation}\label{thermos}
 Z(\rho,\phi)=1+\sum_{n=1}^{\lfloor\Lambda\rfloor}\frac{\rho^n}{\prod_{k=1}^n \phi(k)}.  
 \end{equation}
Note that \eqref{thermos} converges thanks to \eqref{cond_conv}. 
In the following, we will refer to distribution \eqref{extendedPO} as the extended COM-Poisson distribution (eCOM-Poisson) with characteristic couple $(\rho,\phi)$, due to its analogy to the COM-Poisson that appears as a special case for $\phi(n)=n^\delta$, $\delta>0$. 

In the next sections \ref{sec:bern} and \ref{fff} we will analyze some possible choices for the function $\phi$.

\section{\label{sec:bern}
Bernstein functions and their compositional inverses}
We recall that a function $f:(0,\infty)\rightarrow \mathbb R$ is a Bernstein function if $f\in \mathcal C^\infty$, $f(\lambda)\geq 0$, $\forall \: \lambda>0$, and 
\begin{equation}
\label{deriv_bern}
 (-1)^{n-1}f^{(n)}(\lambda)\geq 0, \quad \forall \: n\in \mathbb N^* = \{1,2,\dots\},
 \quad  \forall \lambda>0,
\end{equation} 
where $f^{(n)}$ denotes the $n$-th derivative of $f$. The space of Bernstein functions is usually denoted by $\mathcal{BF}$. Correspondingly, $\mathcal{BF}_\mathrm{b}$ denotes the space of bounded Bernstein functions. Furthermore we denote by $\mathcal{BF}^0$ the space of Bernstein functions $f$ such that $\lim_{\lambda \to 0+} f(\lambda) = 0$. 
We will only recall in this section the results that will be directly used in the following.
For an extensive discussion on the properties of this class of functions
see \cite{schilling2012bernstein}.

A further characterization of a Bernstein  function $f:(0,\infty)\rightarrow \mathbb R$ is given by the following representation:
\begin{equation}\label{integral}
    f(\lambda)=a+b\lambda+\int_{(0,\infty)}(1-e^{-\lambda t})\,\mu(\mathrm dt)
\end{equation}
where $a,b\geq 0$ and $\mu$ is a measure on $(0,\infty)$ satisfying $\int_{(0,\infty)}\min(1,t)\mu(\mathrm dt)<\infty$.
In particular, $f$ is said to be a complete Bernstein function $(f\in \mathcal{CBF})$ if $\mu$ has a completely monotone density with respect to the
Lebesgue measure.

Let us now recall some known properties of the functions mentioned above.
\begin{Proposition}
i) A function $f \in \mathcal{BF}_\mathrm{b}$ if, and only if, in \eqref{integral} $b=0$ and $\mu(0,\infty)<\infty$;
    ii) If $f \in \mathcal{CBF}$, $f\ne 0$ then the functions $\frac{\lambda}{f(\lambda)}$,
    $1/f\left(\frac{1}{\lambda}\right)$ and $\lambda f\left(\frac{1}{\lambda}\right)$ belong to $\mathcal{CBF}$.
\end{Proposition}
A function $f\in \mathcal{BF}$ is said to be a special Bernstein function, and we write $f\in \mathcal{SBF}$, if the
function $\lambda/f(\lambda) \in \mathcal{BF}$.
It follows that if $f \in \mathcal{SBF}$, then $\lambda/f(\lambda) \in \mathcal{SBF}$.

Let $f$ be a positive function in $\mathcal{ BF}^0$.
Then we define its (compositional) inverse $h$ as
\begin{eqnarray}
\label{def_inv}
h(s) :=\inf \{ x>0 \colon f(x)> s \}, \qquad 0<s<\lim_{\lambda \to \infty}f(\lambda) = \Lambda. 
\end{eqnarray}
Plainly, $h\colon (0,\Lambda) \rightarrow \mathbb{R}$ are positive, convex, $\mathcal{C}^\infty$ functions. 
We will write $\mathcal{IBF}$ for the space of such inverse functions.
Correspondingly, we have,
\begin{align}
    f(\lambda) := \inf\{x\in(0,\Lambda)\colon h(x)> \lambda \}.
\end{align}
Notice that $\Lambda = \lim_{\lambda \to \infty}f(\lambda)$ is also the point at which $\lim_{s \to \Lambda -} h(s) = \infty$.

The $n$-th derivative of $h$ can be derived by recurring at the formula for higher order derivatives of inverse functions (see e.g~\cite{kaneiwa}):
\begin{align*}
h^{(n)}(s)&=\frac{(-1)^{n-1}}{[f^{(1)}(h(s))]^{2n-1}} \!\!\!\!
\sum_{\substack{s_1+s_2+\ldots s_n=n-1 \\ s_1+2\cdot s_2+\ldots+n\cdot s_n=2n-2}} \!\!\!\!
\frac{(-1)^{s_1}(2n-s_1-2)![f^{(1)}(h(s))]^{s_1}\dots[f^{(n)}(h(s))]^{s_n}}{(2!)^{s_2}s_2!\dots(n!)^{s_n}s_n!} \\
& = \frac{1}{f^{(1)}(h(s))} \mathcal{C}_{2n-2,n-1}\left( \frac{1}{f^{(1)}(h(s))}, - \frac{f^{(2)}(h(s))}{f^{(1)}(h(s))}, \dots, - \frac{f^{(n)}(h(s))}{f^{(1)}(h(s))}  \right),
\end{align*}
where
\begin{align*}
    \mathcal{C}_{h,k} & (x_1, \dots, x_{h-k+1}) \\
    & = \hspace{-1cm} \sum_{\substack{j_1+j_2+\ldots+j_{h-k+1}=k \\ j_1+2\cdot j_2+\dots+(h-k+1)j_{h-k+1}=h}}
    \hspace{-0.5cm} \frac{1}{\binom{h}{j_1}}
    \frac{h!}{j_1!j_2!\dots j_{h-k+1}!}\left[\frac{x_1}{1!}\right]^{j_1}\left[\frac{x_2}{2!}\right]^{j_2}\hspace{-0.3cm}\dots\left[\frac{x_{h-k+1}}{(h-k+1)!}\right]^{j_{h-k+1}}
\end{align*}
are weighted exponential Bell polynomials with weights $\binom{h}{j_1}^{-1}$.
\begin{rmk}
    Note that a characterization of $h$ through the signs of its derivatives cannot be given as it is done in the case of Bernstein functions. As an example take the function $h(s)=s^{\beta}$, $\beta>1$.
\end{rmk}

\begin{exmp}[Pairs of Bernstein and corresponding inverse Bernstein functions]
\begin{equation}\label{h_1}
    f(\lambda)=\lambda^\alpha; \qquad h(s)=s^{1/\alpha}, \quad \alpha \in (0,1);
\end{equation}

\begin{equation}\label{h_2}
    f(\lambda)=\frac{a\lambda}{\lambda+1}, \quad a \in (0,\infty);
    \qquad h(s)=\frac{s}{a-s}, \quad s \in (0,a);
\end{equation}

\begin{equation}\label{h_3}
    f(\lambda)= (\lambda+1)^\alpha -1; \qquad
    h(s)=(s+1)^{1/\alpha}-1,  \quad \alpha \in (0,1);
\end{equation}

\begin{equation}\label{h_4}
    f(\lambda)=\sqrt{\frac{a\lambda}{\lambda+1}},
    \quad a \in(0,\infty); \qquad h(s)=\frac{s^2}{a-s^2},
    \quad s \in (0,\sqrt{a});
\end{equation}

\begin{equation}\label{h_5}
    f(\lambda)= \mathcal{W}(\lambda); \qquad
    h(s)=se^s, 
\end{equation}
where $\mathcal{W}$ is the Lambert function on the positive real line;

\begin{equation}\label{h_6}
    f(\lambda)= \log(1+\lambda^\alpha); 
    \qquad h(s)=(e^s-1)^{1/\alpha}, \quad \alpha\in (0,1);
\end{equation}

\begin{equation}\label{h_7}
    f(\lambda)= \log(\cosh(\sqrt{2 \lambda})); 
    \qquad h(s)=\frac{1}{2}\log^2(e^s+\sqrt{e^{2s}-1});
\end{equation}

\begin{equation}\label{h_8}
   f(\lambda)=\log\left(1+\frac{\lambda}{a}\right); 
   \qquad h(s)=a(e^s-1), \quad a>0.
\end{equation}
\end{exmp}

\begin{exmp}[Construction based on L\'evy--Laplace exponents]
A large set of inverses of Bernstein functions emerges as follows. The class of L\'evy--Laplace exponents has been considered in \cite{jedidi,Bertoin2004SomeCB} and it is defined as the set
\begin{equation}
    \mathcal{LE}=\left\{ \Psi: \Psi(s)=a+bs+cs^2+\int_{(0,\infty)}(e^{-sx}-1+sx)\nu(\mathrm{d}x), \quad a,b,c \geq 0\right\},
\end{equation}
where $\nu$ is a positive measure on $(0,\infty)$ such that $\int_{(0,\infty)}(x \wedge x^2)\nu(\mathrm{d}x)<\infty$ and $(a,b,c,\nu)$ is the quadruple characteristics of $\Psi$.

If $\Psi \in \mathcal{LE}$ and $a = \Psi(0) = 0$, then $\Psi$ is the Laplace exponent of a spectrally negative L\'evy
process $Z = (Z_t)_{t\geq 0}$,
that is a L\'evy process with no positive jumps that do not drift to $-\infty$ (see \cite{kyprianou2014fluctuations}). Thus, it holds
\begin{equation}
    e^{t\Psi(\lambda)} = \mathbb E[e^{\lambda Z_t}], \qquad t, \lambda \geq 0.   
\end{equation}
Moreover, $\Psi(\lambda)=\lambda f(\lambda)$ where $f$ is a Bernstein function \cite{patie_jacobi}.

The class of  functions in $\mathcal{LE}$ with $a=0$ is denoted by $B_3$ in \cite{Bertoin2004SomeCB} and is also known as the class of branching mechanisms for (sub)critical continuous state branching processes \cite{gall1999spatial}.

From Proposition 9 in \cite{Bertoin2004SomeCB},
every function in $B_3$ is such that its inverse is a Bernstein function, 
implying that every function $\Psi$ belonging to $\mathcal{LE}$ with $a=0$ can be considered as the inverse function of a Bernstein function $\widetilde f$.
We stress that $\widetilde f$ is not usually the Bernstein function associated to the process $Z$.

One can give the following interpretation:
if $\widetilde f$
is the the inverse of $\Psi \in B_3$, the Bernstein
function $[\frac{\mathrm d}{\mathrm d \lambda}\widetilde f(\lambda)]^{-1}$
is the exponent of the subordinator defined as the inverse of the local time
at zero of $Z$.
\end{exmp}

\begin{rmk}
\label{rmk:conv}
  Let us discuss condition \eqref{cond_conv} for Bernstein and Inverse Bernstein functions.
  If $\phi \in \mathcal{IBF}$ and $\Lambda=\infty$, then $\phi$ grows at least linearly at infinity. Hence,  \eqref{cond_conv} holds for all $\rho \in \mathbb R^+$.
  If $\phi \in \mathcal{BF}$ such that $\phi(\infty) =\infty$, then \eqref{cond_conv} holds true  for all $\rho \in \mathbb R^+$ due to the following asymptotic expansion of the Bernstein-Gamma function  (see \cite{MR4320772}, Theorem 5.1), 
  $$W_{\phi}(n)=\prod_{j=1}^{n-1}\phi(j) \asymp C_{\phi}\frac{1}{\phi(n)\sqrt{\phi(n+1)}}e^{\int_1^{n+1}\log \phi(u)du}.$$  
  If $\phi \in \mathcal{BF}$ is bounded by $\Lambda$, then by examining the limit of the ratio of consecutive terms,
   \eqref{cond_conv} holds  true for all $\rho \in (0,\Lambda)$.
   \end{rmk}

\section{Some properties of inverse Bernstein functions}\label{fff}

\subsection{Log-concave and exp-convex functions}

We noted in the previous section that
the inverse functions defined as in  \eqref{def_inv} are convex. We discuss here further properties.

We recall that a  positive function  $g\colon I\subset \mathbb R\rightarrow \mathbb R^+$,  is log-convex (-concave) if $\log g$ is a convex (concave) function. Similarly, $g$ is  exponentially-convex (-concave) 
if $\exp g$ is a convex (concave) function.
A characterizing property of twice-differentiable log-convex (-concave) functions is that $g''(x)g(x)-(g'(x))^2$ is non-negative (non-positive) \cite{webster}.
Finally, we recall that a function $g$ is eventually log-convex if there exists $m \in \mathbb R$ such that $\log g$ is convex in $(m,\infty)$.\\
The following characterizing inequalities for log-convex functions $g: I\rightarrow \mathbb R^+$ hold:
\begin{equation}
\label{logconv_1}
    g(\lambda x+\mu y)\leq g^\lambda(x)g^\mu(y)
    \end{equation}
     \begin{equation}
     \label{logconv_2}
     g^\lambda(x) g^\mu(y) \leq  \lambda g(x)+\mu g(y)
     \end{equation}
whenever, $x,y\in I$, $\lambda, \mu>0$ and $\lambda+\mu=1$
and
      \begin{equation}
      \label{logconv_3}
     g^{z-x}(y)\leq g^{z-y}(x)g^{y-x}(z)
     \end{equation}
for $x,y,z \in I$ and $x<y<z$.

We have the following results.
\begin{lem}\label{lemma_exp}
  Let $g(x)$, $x\in \mathbb R^+$, be a twice-differentiable exponentially-concave (-convex) function. Then, 
  \begin{equation}
      g''(x)+\left(g'(x) \right)^2\leq (\geq) \ 0, \qquad \forall \, x\in \mathbb R^+.
  \end{equation}
\end{lem}
\begin{proof} We prove the exponentially-concave case. Clearly, $\frac{d^2}{dx^2}e^{g(x)}\leq 0$
    and the statement follows after straightforward calculations.
    The exponentially-convex case is proved analogously. 
\end{proof}

\begin{lem}\label{lemma_log}
  Let $f$ be an exponentially-concave  function in $\mathcal{BF}^0$. Then, its compositional inverse $h$ is eventually log-convex.
\end{lem}
\begin{proof}
We want to prove that
$h''(s)h(s)-(h'(s))^2\geq 0$,
that is a characterizing property of log-convex functions.

Since $h(f(\lambda)))=\lambda$,
we have that
\begin{eqnarray*}
    f'(\lambda)&=&\frac{1}{h'(f(\lambda))} \\
    f''(\lambda)&=&-\frac{h''(f(\lambda))}{\left[h'(f(\lambda))\right]^3}
\end{eqnarray*}
where $h'(f(\lambda))> 0$ as  $h'(f(\lambda))=1/f'(\lambda)$ and $f$ is a Bernstein function.
From Lemma \ref{lemma_exp} we have that $f''(\lambda)+\left(f'(\lambda) \right)^2\leq 0, \, \forall \lambda$, that, written in terms of the derivatives of $h$, gives
\begin{align}-\frac{h''(f(\lambda))-\left[h'(f(\lambda))\right]^2}{\left[h'(f(\lambda))\right]^3}\leq 0.\end{align}
Since $h$ is a non-decreasing function, we conclude that
\begin{align}
    h''(f(\lambda))-\left[h'(f\lambda))\right]^2\geq 0
\end{align}
and thus
\begin{align}
    h''(f(\lambda))h(f(\lambda))-\left[h'(f(\lambda))\right]^2=h''(f(\lambda))\lambda-\left[h'(f(\lambda))\right]^2\geq 0, \quad \forall \lambda \geq 1
\end{align}
proving that $h$ is eventually log-convex with $m=f(1)$.
\end{proof}

\subsection{\label{sec:level5b}Bernstein-Gamma function and its analogue for inverse Bernstein functions}

In the following we write, for a function $\phi$ as in Section \ref{sec:level2},
\begin{equation}
\label{eq:product-Wphi}
W_\phi(1)=1, \qquad W_{\phi}(n)=\prod_{k=1}^{n-1} \phi(k), \quad n \in \{2,3,\dots\}.
\end{equation}
In particular, if $\phi\in\mathcal{BF}$, then  $W_{\phi}$ stands for the unique solution, in the space of positive definite functions, to the functional equation
\begin{align}\label{equazionegamma}
    W_{\phi}(n+1)=\phi(n)W_{\phi}(n)
\end{align}
with $W_{\phi}(1)=1$, and we refer to \cite{patie2018bernstein, MR4320772, hirsch_yor} for a thorough account on this set of functions that generalizes the gamma function, which appears as a special case when $\phi(n)=n$.
Further, Theorem 4.1 of \cite{webster} extends $W_\phi$ to the real line. 

If $\phi$ is an eventually log-convex function such that \cite{hooshmand2022remarks}
\begin{equation}
\label{cond_webster}
\lim_{n\rightarrow \infty} \frac{\phi(n)}{\phi(n+1)}=1,  
\end{equation}
then, the following theorem enables us to extend $W_\phi$ to the real line as well.

\begin{thm}\label{webster_conv}
Let $\phi$ be an eventually log-convex
function  on $\mathbb  R^+$ satisfying condition \eqref{cond_webster}. Then, there exists a function $\widetilde W_\phi$ satisfying the functional equation $\widetilde W_\phi(x +1) = \phi(x)\widetilde W_\phi (x)$, $x \in \mathbb R_+$,
with the initial condition $\widetilde W_\phi(1) = 1$.
Moreover
\begin{equation}\label{w_tilde}
  \widetilde W_\phi(x)=  \lim_{n\rightarrow \infty} \frac{\phi(n)\cdots \phi(1)\phi^x(n)}{\phi(n+x)\cdots \phi(x)}, \qquad x>0.
\end{equation}
\end{thm}   
\begin{proof}
The proof follows  the lines of Theorem 4.1 in \cite{webster}, properly adapted to the case of log-convex functions as sketched below.\\
To prove the existence of $ \widetilde W_\phi(x)$, let us define for every $n\in \mathbb N$ the functions $\widetilde W_n:\mathbb R^+ \rightarrow \mathbb R^+$ such that
$$\widetilde W_n(x)=\frac{\phi(n)\cdots \phi(1)\phi^x(n)}{\phi(n+x)\cdots \phi(x)}, \qquad x>0.$$
It is easy to show the following recurrence relations: 
\begin{equation}   
\label{rel:Wn}
\widetilde W_{n+1}(x)=\frac{\phi^{x+1}(n+1)}{\phi(n+1+x)\phi^x(n)}\widetilde W_n(x), 
\qquad 
\widetilde W_{n}(x+1)=\frac{\phi(n)\phi(x)}{\phi(n+1+x)}\widetilde W_n(x).
\end{equation}
By hypothesis, $\phi$ is log-convex in $(m,\infty)$ for some $m\geq 0$. Then, from \eqref{logconv_3} with $z=n+x+1$, $x=n$, $y=n+1$ and $n>m$ we have
$$\phi^{x+1}(n+1) \leq \phi^x(n)\phi(n+x+1)$$
showing that the sequence $(\widetilde W_n)_{n\in \mathbb N}$ is eventually decreasing. Moreover, the sequence is bounded below by zero, then, by the monotone convergence theorem, it converges.
Hence, the function in \eqref{w_tilde} is well-posed and satisfies 
$\widetilde W_\phi(x +1) = \phi(x)\widetilde W_\phi (x)$ 
(see also Lemma 2.1 of \cite{hooshmand2022remarks}).
Finally, the initial condition is
\begin{equation}
    \widetilde W(1)=\lim_{n\rightarrow \infty} \frac{\phi(n) \cdots \phi(1)\phi(n)}{\phi(n+1)\cdots \phi(1)}=
    \lim_{n\rightarrow \infty} \frac{\phi(n) }{\phi(n+1)}=1,
\end{equation}
where the last equality follows from \eqref{cond_webster}.
\end{proof}

\begin{rmk}
    Recalling Lemma \ref{lemma_log}, we have that inverses of exponentially-concave $\mathcal{BF}^0$ functions, such that $\Lambda=\infty$,  satisfy Theorem \ref{webster_conv}. 
\end{rmk}
\begin{rmk}
    The solution of \eqref{equazionegamma} for both $\phi$ convex or concave is in general not unique, see for instance \cite{webster} and references therein.  
\end{rmk}

\section{Some properties of the eCOM-Poisson distribution}\label{sec:level4b}

Let us consider a discrete random variable $X$ with eCOM-Poisson distribution with characteristic couple $(\rho,\phi)$.
We first observe that
the  ratio of successive probabilities 
\begin{align}\label{ratedecay}
    \frac{\mathbb P(X=n-1)}{\mathbb P(X=n)}=\frac{\phi(n)}{\rho}, \qquad n \in \mathbb{N}^*,
\end{align}
gives us the rate of decay of the tail of the distribution.
Moreover, we also have
\begin{align}
    \frac{\mathbb P(X=n-s)}{\mathbb P(X=n)}=\frac{\phi(n)\phi(n-1)\dots\phi(n-s+1)}{\rho^s}, \qquad n \ge s.
\end{align}

To evaluate the moments of $X$, following \cite{macci_pacchiarotti_villa_2022} (see also \cite{johnson}) where a power series distribution is considered, we make use of the probability generating function of $X$, i.e. 
\begin{equation}\label{eq:pgf}
\mathbb{E}u^X=\frac{Z(u\rho,\phi)}{Z(\rho,\phi)}, \qquad  |u|\leq 1.
\end{equation}
\begin{rmk}
    If $\phi\in \mathcal{BF}$ expression \eqref{eq:pgf} can be written as 
    \begin{eqnarray}
        \mathbb Eu^X= \frac{ \  _1F_{1}(1;\phi;\rho u)}{_1F_{1}(1;\phi;\rho) },
    \end{eqnarray}
    where ${}_1F_1(1;\phi;\cdot)$ is the extension of the hypergeometric function ${}_1F_1$ given in \cite{DPS2022}. From this expression one can calculate summary statistics of $X$ using properties of this function.
    Note the analogy with the case of the COM-Poisson probability generating function (see \cite{nadarajah2009useful}).
\end{rmk}

The factorial moments of order $s\in \mathbb N^*$ of $X$ are given by
\begin{equation}
    m_s =\frac{\mathrm d^s}{\mathrm du^s}\mathbb Eu^X\Big |_{u=1}
    =\mathbb E\left[X(X-1)(X-2)\cdot \ldots \cdot (X-s+1) \right]
    =(-1)^s\mathbb E(-X)_s,
\end{equation}
where $(y)_s = y(y+1)\cdot \dots \cdot (y+s-1)$ is the rising factorial (also known as the Pochhammer symbol).
Now, by virtue of the Vi\`ete--Girard formulae, that relate the coefficients of a polynomial to sums and products of its roots, we expand  $X(X-1)(X-2)\cdots (X-s+1)$ as
\begin{align}
    m_s=\sum_{r=1}^s(-1)^{s-r} e_{r} \, \mathbb{E}X^r,
\end{align}
where
\begin{align}
    e_{r}=e_{r}(l_1,l_2,\ldots,l_r)=\sum_{1\leq l_1\leq l_2\leq \cdots \leq l_r \leq s}l_1l_2\cdot \ldots \cdot l_r.
\end{align}
By using the relation between moments and factorial moments we have
\begin{align}
    \mathbb{E}X^s=\sum_{j=0}^{s}(-1)^j\stirling{s}{j}\mathbb{E}(-X)_j
\end{align}
where $\stirling{s}{j}=\frac{1}{j}\sum_{m=0}^{j}(-1)^{j-m}\binom{j}{m}m^s $
are the Stirling numbers of the second kind. Considering that
\begin{align}
    \mathbb E(-X)_j= \frac{1}{Z(\rho,\phi)}\sum_{k=1}^{\infty}(-k)_j\frac{\rho^k}{\prod_{i=1}^k \phi(i)},
\end{align}
in which we used that $(0)_j = 0$,
it follows
\begin{equation}  \label{moments}
\mathbb EX^s=\frac{1}{Z(\rho,\phi)}\sum_{j=0}^{s}(-1)^j\stirling{s}{j}\sum_{k=1}^{\infty}(-k)_j\frac{\rho^k}{\prod_{i=1}^k \phi(i)}, \qquad s \in \mathbb{N}^*.
\end{equation}
Writing $(-k)_j=(-1)^jk!/(k-j)!$, we identify in \eqref{moments} the operator $(\rho D)$  characterized by the well known formula 
\cite{HWG78}
\begin{equation}\label{s_derivative}
    (\rho D)^sf(\rho)=\sum_{j=0}^s \stirling{s}{j}\rho^jD^jf(\rho),
\end{equation}
where $D=\mathrm d/\mathrm d\rho$, and $f$ is a suitable function.
The previous steps prove the following proposition.
\begin{Proposition}\label{prop_moments}
    Let $X \sim \text{eCOM-Poisson($\rho, \phi$)}$. Then,
    \begin{eqnarray} 
      \mathbb EX^s &=& \frac{1}{Z(\rho,\phi)}(\rho D)^s Z(\rho,\phi), \quad s\in \mathbb N^*,  \\
      m_s &=&\frac{1}{Z(\rho,\phi)}\rho^s D^s Z(\rho,\phi), \quad s\in \mathbb N^*.
    \end{eqnarray}
\end{Proposition}
In general, the normalizing function $Z(\rho,\phi)$ does not permit neat closed-form expressions for the statistical quantities related to this distribution. However, asymptotic results on the product $\prod_{i=1}^k \phi(i)$ are available in the case of Bernstein functions, see Theorem 4.2 in \cite{patie2018bernstein} (see also \cite{minchev2023asymptotics}).
A closed form for a sort of a generalization of the factorial moments is presented in the following remark.
\begin{rmk}
    For $X\sim\text{eCOM-Poisson($\rho,\phi$)}$, by direct calculation, we have
    \begin{align}\label{phifac}
        {}_\phi m_s = \mathbb{E}\bigl[ \phi(X)\phi(X-1)\dots \phi(X-s+1) \bigr] = \rho^s, \qquad s \in \mathbb{N}^*,
    \end{align}
    which remarkably coincide with the factorial moments of a Poisson random variable of parameter $\rho$.
    Further, \eqref{phifac}  specialize to the factorial moments if $\phi$ is the identity function (Poisson case). Further, if $s=1$ we have
    \begin{align}\label{phifac2}
        \mathbb{E} \phi(X) = \rho.
    \end{align}
\end{rmk}

\begin{rmk}
    By the extended Markov's inequality for arbitrary non-negative and non-decreasing functions, we have, for every $a>0$,
    \begin{align}
        \mathbb{P} (X \ge a) \le \frac{\mathbb{E}\phi(X)}{\phi(a)} = \frac{\rho}{\phi(a)},
    \end{align}
    where we have used $\eqref{phifac2}$. For $a \in \mathbb{N}^*$, considering \eqref{ratedecay}, we have
    \begin{align}
        \mathbb{P}(X \ge a) \le \frac{\mathbb{P}(X=a)}{\mathbb{P}(X=a-1)}.
    \end{align}    
\end{rmk}

In the following we present a few examples of eCOM-Poisson distributions for different choices of the function $\phi$  belonging to either the class of Bernstein functions or that of their  inverses.

\begin{exmp} \label{ex_1}
Let $X \sim \text{eCOM-Poisson($\rho, \phi$)}$ such that $\rho \in (0,1)$ and $\phi$ is the Bernstein function $\phi(r)=\frac{r}{r+1}$.
Condition $\rho \in (0,1)$ guarantees that \eqref{cond_conv} is satisfied.
Then $\prod_{r=1}^{n}\phi(r)=\frac{1}{n+1}$ and  $Z(\rho,\phi)=\sum_{n=0}^{\infty}(n+1)\rho^n={(1-\rho)^{-2}}$.
From Proposition \ref{prop_moments}, we obtain

\begin{align*}
\mathbb EX^s & = (1-\rho)^2(\rho D)^s(1-\rho)^{-2} 
= (1-\rho)^2 (\rho D)^s \sum_{j=1}^\infty j\rho^{j-1} \\
&= (1-\rho)^2 \sum_{j=0}^\infty (j+1)j^s \rho^{j} 
= (1-\rho)^2 [\text{Li}_{-s-1}(\rho) + \text{Li}_{-s}(\rho)],
\end{align*}
where  $\text{Li}_{-s}$ is the polylogarithm function. 
In particular, the first  moment of $X$ and its variance  read 
 \begin{equation}
   \mathbb EX=\frac{2\rho}{1-\rho}, \qquad 
   \mathbb V\text{ar} X=\frac{2\rho}{(1-\rho)^2}.
 \end{equation}
 \end{exmp}
 \begin{exmp}
Consider the inverse Bernstein function $\phi(r)=(r+1)^2-1=r(r+2).$ Then $\prod_{r=1}^{n}\phi(r)=\frac{n!(n+2)!}{2}$ and 
Proposition \ref{prop_moments} can be applied for $Z(\rho,\phi)=\sum_{n=0}^{\infty}\frac{2\rho^n}{(n+1)(n+2)n!^2}.$ 
Clearly, condition \eqref{cond_conv} holds true.
This function satisfies the Euler differential equation of the second order:
\begin{equation}
    \rho^2 S''(\rho)+4\rho S'(\rho)+2S(\rho)=2C_0(\rho),
\end{equation}
where $C_0(\rho)=\sum_{n=0}^{\infty}\frac{\rho^n}{n!^2}$ is the Tricomi function.
\end{exmp}

\begin{exmp}[COM-Poisson]
Consider the function $\phi(r)=r^\delta$, for $\delta >0$.
Then $\prod_{r=1}^{n}\phi(r)=n!^\delta$, \ $Z(\rho,\phi)=\sum_{n=0}^{\infty}\frac{\rho^n}{n!^\delta}=C_\delta(\rho)$ is the Le Roy function, introduced in \cite{leroy} (see also \cite{MR3171991,gerhold2012asymptotics,simon2022remark} for a more general Le Roy-type function).
Note that if $\delta \in (0,1)$, then $\phi(r)\in \mathcal{BF}$,  if $\delta >1$, then $\phi(r)\in \mathcal{IBF}$.
There does not exist closed form formulae for the moments of the COM-Poisson distribution, although recurrent relations \cite{daly}, \cite{Shmueli2005} or asymptotic expansions \cite{gaunt2019asymptotic} are available. An alternative formula can be obtained from Proposition \ref{prop_moments}: 
\begin{align}
    \mathbb{E}X^s=\frac{(\rho D)^s C_\delta(\rho)}{C_\delta(\rho)}.
\end{align}
\end{exmp}

\begin{exmp}
Consider the Bernstein function $\phi(r)=\frac{r+a}{r+b}$, $a,b>0$.
Then, $\prod_{r=0}^{n}\phi(r)=\frac{(a)_{n+1}}{(b)_{n+1}}$  and
\begin{equation}
    Z(\rho,\phi)=\frac{b}{a}\sum_{n=0}^{\infty}\frac{(1)_n (b+1)_n}{(a+1)_n}\frac{\rho^n}{n!}=\frac{b}{a}\ _2F_1(1,b+1;a+1;\rho),
\end{equation}
where ${}_2F_1$ is the Gauss hypergeometric function.
From Proposition \ref{prop_moments}, we obtain
\begin{eqnarray}
    \mathbb EX^s
 =\frac{(\rho D)^s \ _2F_1(1,b+1;a+1;\rho)}{_2F_1(1,b+1;a+1;\rho)},
\end{eqnarray}
and using the well-known differentiation formula (see 15.2.2 in \cite{abramowitz1968handbook}):
\begin{equation}
 D^j  {}_2F_1(\mu,\nu;\lambda;\rho)=\frac{(\mu)_j (\nu)_j}{(\lambda)_j}\\ _2F_1(\mu+j,\nu+j;\lambda+j;\rho),  
\end{equation}
we get the expression for the factorial moments
\begin{eqnarray}
    m_s=\rho^s\frac{ (1)_s (b+1)_s}{(a+1)_s} \frac{_2F_1(1+s,b+1+s;a+1+s;\rho)}{  _2F_1(\mu+j,\nu+j;\lambda+j;\rho)}.    
\end{eqnarray}
\end{exmp}

Finally, we consider an example in which 
$\phi$ is neither a Bernstein nor an inverse Bernstein function.

\begin{exmp}
Consider the function $\phi(r)=\frac{r^2+r+1}{r^2-r+1}$. For $\rho \in (0,1)$ condition \eqref{cond_conv} is satisfied.
Then, $\prod_{r=1}^{n}\phi(r)=n^2+n+1 $ and
\begin{align}
    Z(\rho,\phi)=\sum_{n=0}^{\infty}\frac{\rho^{n}}{n^2+n+1}=\sum_{n=0}^{\infty}\frac{\rho^{n}}{(n+1/2)^2+3/4}.
\end{align}
The latter quantity is a Mathieu series \cite{tomovskipoganj2011}. Consider formula (5) in (\cite{watson1966handbook}, page 386), due to Gegenbauer,
\begin{equation}
    \frac{1}{(p^2+a^2)^{\mu}}=\frac{\sqrt{\pi}}{(2a)^{\mu-1/2}\Gamma(\mu)}\int_0^{\infty}e^{-px}x^{\mu-1/2}J_{\mu-1/2}(ax) \mathrm dx \qquad a\in\mathbb{R},\mu,p>0,
\end{equation}
where $J_{\mu-1/2}(ax)$ is a Bessel function of the first kind.
Taking $p=n+1/2, \mu=1, a=\frac{\sqrt 3}{2},$ we get
\begin{equation}
    \frac{1}{(n+1/2)^2+\frac{3}{4}}=3^{-1/4}\sqrt{\pi}\int_0^{\infty}e^{-(n+1/2)x} \sqrt{x} \, J_{1/2}\left({\frac{\sqrt 3}{2}x}\right) \mathrm dx, \qquad n \in\mathbb{N}.
\end{equation}
From here, we get the following integral representation
\begin{equation}
 Z(\rho,\phi)=\frac{2}{\sqrt 3}\int_0^{\infty}\frac{ e^{\frac{x}{2}}}{e^x-\rho} \sin\left({\frac{\sqrt 3}{2}}x\right) \mathrm dx.
\end{equation}
\end{exmp}


\section{Dispersion properties}\label{disssp}

In this section we consider $\Lambda=\infty$. Recall 
the Poisson probability mass function  $P_n(\rho)$, $\rho >0$, reads
\begin{equation}
P_n^{\textup{id}}(\rho)=P_n(\rho)=\frac{\rho^n}{n!}e^{-\rho},\qquad n\in \mathbb{N},
\end{equation}
where $\textup{id}$ is the identity function.
Clearly,
\begin{equation}
 P_n^{\phi}(\rho)=\frac{\frac{  P_n(\rho)}{\prod_{j=1}^n \phi(j)}}{\sum_{k=0}^{\infty}\frac{ P_k(\rho)}{\prod_{j=1}^k \phi(j)}}, \qquad n\in \mathbb{N},   
\end{equation}
and hence it can be viewed as a weighted Poisson distribution,
\begin{align}
    P_n^{\phi}(\rho) = \frac{\rho^n w(n)e^{-\rho}}{n!\,\mathbb{E}w(Y)}, \qquad n \in \mathbb{N},
\end{align}
where $Y\sim \text{Poisson}(\rho)$ and $w(n) = n! / \prod_{j=1}^n \phi(j)$ is the weight function. Note that $w(n)$ does not depend on $\rho$.

In order to study overdispersion and underdispersion properties of the model we consider Theorem 3 of \cite{MR2388011} together with the corollary following Theorem 4 of the same paper, and thus we study log-convexity of the weight function $w(n)$, that is convexity of
\begin{align}
    \log w(n) = \log \frac{n!}{\prod_{k=1}^n \phi(k)}.
\end{align}
It is well-known that a sequence of real numbers $(a_n)_{n \in \mathbb{N}}$ is convex if it satisfies the inequality
\begin{align}
    2 a_n \le a_{n-1} + a_{n+1}, \qquad \text{for every }n \in \mathbb{N}^*.
\end{align}
Hence, $w(n)$ is log-convex, and hence the model exhibits overdispersion, if
\begin{align}
    2 \log \frac{n!}{\prod_{k=1}^n \phi(k)} - \log \frac{(n-1)!}{\prod_{k=1}^{n-1} \phi(k)} - \log \frac{(n+1)!}{\prod_{k=1}^{n+1} \phi(k)} \le 0, \qquad \text{for every }n \in  \mathbb{N}^*.
\end{align}
This is equivalent to
\begin{align}
    \log \biggl(\frac{n}{n+1} \frac{\phi(n+1)}{\phi(n)}\biggr)\le 0, \qquad \text{for every }n \in  \mathbb{N}^*,
\end{align}
leading, in turn, to
\begin{align}
\label{condition1}
    \frac{n+1}{n} \ge \frac{\phi(n+1)}{\phi(n)}, \qquad \text{for every }n \in  \mathbb{N}^*.
\end{align}
Equivalently, for the dispersion function $d(\lambda) = \lambda/\phi(\lambda)$, $\lambda \in (0,\infty)$,
\begin{align}
    d(n+1) \ge d(n), \qquad \text{for every }n \in  \mathbb{N}^*,
\end{align}
i.e.\ $d(n)$ is non-decreasing. This leads to the following theorem.

\begin{thm}\label{th:ecom_disp}
    Let $X$ be the random variable with probability mass function \eqref{pn_COM}. Then $X$ is overdispersed (underdispersed) if $d(n) = n/\phi(n)$ is non-decreasing (non-increasing).
    
\end{thm}

\begin{cor}
    \label{cor_sbf}
If $\phi \in \mathcal{BF}$ (resp.\ $\mathcal{IBF}$), the random variable $X$ is overdispersed (resp.\ underdispersed). 
\end{cor}

\begin{cor}\label{cor_ndecr} 
    If $\phi \in \mathcal{C}^1$ is non-decreasing, then $X$ is overdispersed (underdispersed) if for every $\lambda \in (0,\infty)$, $\phi' (\lambda)\le (\ge) \: \phi(\lambda)/\lambda$.
\end{cor}

%

%


\begin{rmk}
    Note that, if $\phi\in \mathcal{SBF}$, so is $d$, and vice versa. Further, consider that $w(n) = \prod_{k=1}^n d(k)$. This last comment  leads us to the duality property in which $w(n)$ and $W_\phi(n) = \prod_{k=1}^n \phi(k)$ can be exchanged.
\end{rmk}

\section{A further extension}\label{extension}

Consider a non–negative random variable $X$ having probability mass function

\begin{eqnarray} \label{pn_COM-gen}
P_n^{\phi}(\rho)=\frac{\Gamma(n+\gamma)\rho^n}{n! V_{\phi}(\alpha n+\beta)}\frac{1}{Z_{\alpha,\beta,\gamma}(\rho,\phi)}, \qquad \alpha, \beta, \gamma > 0,\  n\in \mathbb N,
\end{eqnarray}
where
\begin{equation}
\label{7.2}
 Z_{\alpha,\beta,\gamma}(\rho,\phi)=\sum_{k=0}^{\infty}\frac{\Gamma(k+\gamma)\rho^k}{k!V_{\phi}(\alpha k+\beta)}
\end{equation} 
and $V_\phi$ is either the Bernstein-gamma function $W_\phi$ or  $\widetilde W_\phi$ in Theorem \ref{webster_conv}.
If $\alpha=\beta=\gamma=1$ we will write $Z_{1,1,1}(\rho,\phi)=Z(\rho,\phi)$.

\begin{rmk}
    Regarding the convergence of \eqref{7.2}, by examining the limit of the ratio of consecutive terms and using the definition \eqref{w_tilde}  of $V_\phi$ for real argument,
    we get that the normalization series converges 
    for $0<\rho<(\lim_{x\rightarrow \infty}\phi(x))^\alpha$. 
\end{rmk}

\begin{rmk}
    From \eqref{pn_COM-gen} we retrieve the following known special cases.
    \begin{itemize}
        \item If $\gamma = \beta = \alpha = 1$, $X$ is eCOM-Poisson($\rho,\phi$). 
        \item If $\gamma = \beta = \alpha = 1$, $\phi \equiv \textup{id}$, $X$ is Poisson with parameter $\rho$.
        \item If $\gamma = \beta = \alpha = 1$, $\phi(x) = x^\delta$, with $\delta \in \mathbb{R}_+$, $X$ is COM-Poisson with parameters $\rho$ and $\delta$ \cite{conway1962queuing}.
        \item If  $\gamma = \alpha = 1$, $\phi \equiv \textup{id}$, $X$ has the hyper-Poisson distribution \cite{bardwell1964two} with parameters $\rho$ and $\beta$.
        \item If $\gamma = 1$, $\phi \equiv \textup{id}$, $X$ has the alternative Mittag--Leffler distribution (see  \cite{MR2535014,MR3543682}).
        \item If $\phi \equiv \textup{id}$ we obtain the alternative generalized Mittag-Leffler distribution \cite{pogany2016probability} (see also \cite{tomovski2014laplace} for properties of the generalized Mittag--Leffler function).
    \end{itemize}
    Furthermore, note that  distribution \eqref{pn_COM-gen} for $\phi(x) = x^\delta$, $\delta \in \mathbb{R}_+$, does not coincide  with that in Section 3.1 of \cite{cahoy2021flexible} although both based on the principle of generalizing the factorial by a gamma-type function, and choosing the power function (see also \cite{garra2018note} if in addition $\gamma=1$).
\end{rmk}

\begin{Proposition}\label{prop:factorial_generalized}
  The factorial moments of $X$ write
  \begin{equation}\label{fact}
    m_s=\rho^s\frac{Z_{\alpha,s\alpha+\beta,\gamma+s}(\rho,\phi)}{Z_{\alpha,\beta,\gamma}(\rho,\phi)} , \qquad s\in \mathbb N.  
  \end{equation}
  \end{Proposition}
  \begin{proof}
     The result follows directly from the application of Proposition \ref{prop_moments}. 
  \end{proof}
\begin{cor}
    The moments read
    \begin{align}
        \mathbb{E}X^s = \frac{1}{Z_{\alpha,\beta,\gamma}(\rho,\phi)} \sum_{r=1}^s \rho^r \stirling{s}{r} Z_{\alpha,r\alpha+\beta,\gamma+r}(\rho,\phi), \qquad s \in \mathbb{N}.
    \end{align}
    In particular, the first two moments of $X$ read
    \begin{eqnarray}
\mathbb{E}X &=&\rho\frac{Z_{\alpha, \alpha+\beta, \gamma+1}(\rho,\phi)}{Z_{\alpha,\beta,\gamma}(\rho,\phi)}, \label{momento}\\
\mathbb{E}X^2&=& \rho^2\frac{Z_{\alpha, 2\alpha+\beta, \gamma+2}(\rho,\phi)}{Z_{\alpha,\beta,\gamma}(\rho,\phi)}
+\rho\frac{Z_{\alpha, \alpha+\beta, \gamma+1}(\rho,\phi)}{Z_{\alpha,\beta,\gamma}(\rho,\phi)}.\label{momentodue}
    \end{eqnarray}
\end{cor}

\begin{cor}
    If $\gamma = \beta = \alpha = 1$, (eCOM-Poisson case) the factorial moments \eqref{fact} simplify to
    \begin{align}\label{factt}
        m_s = \rho^s \frac{Z_s(\rho,\phi)}{Z(\rho,\phi)}, \qquad s \in \mathbb{N},
    \end{align}
    where $Z_s(\rho,\phi) = Z_{1,s+1,s+1}(\rho,\phi)$. 
    If, in addition,  $\phi(x) = x^\delta$, $\delta \in \mathbb{R}_+$, formula \eqref{factt}, in agreement with (52) of \cite{cahoy2021flexible}, gives the factorial moments of the COM-Poisson distribution:
    \begin{align}
        m_s = \frac{\rho^s}{C_\delta(\rho)} \sum_{k=0}^\infty \frac{\rho^k}{k!} (k+s)!^{1-\delta}, \qquad n \in \mathbb{N},
    \end{align}
    where $C_\delta(\rho)$ is the Le Roy function.
\end{cor}

\begin{rmk}
    The following Tur\'an-type inequality for the function $Z$ holds: 
    \begin{equation}
      Z_{\alpha,\alpha+\beta,\gamma+1}(\rho,\phi)+\rho Z_{2\alpha,\alpha+\beta,\gamma+2}(\rho,\phi)\geq \rho\frac{Z_{\alpha,\alpha+\beta,\gamma+1}^2(\rho,\phi)}{Z_{\alpha,\beta,\gamma}(\rho,\phi)}. 
    \end{equation} 
    This directly follows from the expressions of the first two moments  of $X$ in \eqref{momento} and \eqref{momentodue}.
\end{rmk}

In the two following propositions  we provide conditions for over and underdispersion of $X$. Note that, in this case, $V_\phi$ cannot be expressed in terms of products of functions $\phi$. 
\begin{Proposition}
\label{prop:dispersion}
The random variable $X$ is overdispersed (underdispersed) if and only if
\begin{eqnarray}\label{prop:disp}
    Z_{\alpha,\beta,\gamma}(\rho,\phi)
    Z_{\alpha,2\alpha+\beta,\gamma+2}(\rho,\phi)> (<) \  Z_{\alpha,\alpha+\beta,\gamma+1}^2(\rho,\phi).
\end{eqnarray}
\end{Proposition}
\begin{proof}
    We prove the characterization for the overdispersion case. The same reasoning applies to the case of underdispersion.
    The variable $X$ is overdispersed if and only if $m_2>m_1^2$ which in this case, recalling Proposition \ref{prop:factorial_generalized}, corresponds to \eqref{prop:disp}.
\end{proof}
\begin{rmk}
    Note that, Proposition \ref{prop:dispersion} applies to the special case of the eCOM-Poisson. However, Theorem \ref{th:ecom_disp} provides a condition that is easier to check.
\end{rmk}
Next, we provide a sufficient condition for  overdispersion (underdispersion) that involves the first derivative of the function
\begin{equation}\label{digamma}
   \psi_\phi(y):=\frac{\mathrm{d}}{\mathrm{d y}}\log V_\phi(y), \qquad y>0,
\end{equation}
where, in particular, $\psi_{\textup{id}}$ is the classical digamma function.
\begin{Proposition}
    If
    \begin{equation}
    \label{disp_prime}
        \psi\prime_{\textup{id}}{(y+\gamma)} \geq (\leq) \ \alpha \psi\prime_\phi(\alpha y +\beta)
    \end{equation}
    where $\psi\prime_\phi$ is the first derivative of the function defined in \eqref{digamma}, 
    then $X$ is overdispersed (underdispersed).
 \end{Proposition}
\begin{proof}
  Using Theorem 4 of \cite{MR2388011}, we consider log-convexity of the weight function $w(y)=\Gamma(y+\gamma)/V_\phi(\alpha y+\beta)$, that is 
  \begin{eqnarray}
     \frac{\mathrm{d}^2}{\mathrm{d}y^2}\log 
     \frac{\Gamma(y+\gamma)}{V_\phi(\alpha y+\beta)} = \frac{\mathrm{d}}{\mathrm{d}y}\left[\psi_{\textbf{id}}(y+\gamma)-\alpha \psi_\phi(\alpha y+\beta)
     \right]\geq (\leq) \, 0
  \end{eqnarray}
  from which the statement immediately follows.
\end{proof}
\begin{rmk}
    Note that for $\alpha=1$, \eqref{disp_prime} reduces to
    \begin{equation}
        \frac{\psi\prime_{\textup{id}}{(y+\gamma)}}{\psi\prime_\phi(y +\beta)} \geq (\leq) \ 1. 
    \end{equation}  
    If, in addition, $\phi \equiv \textup{id}$, then \eqref{disp_prime} coincides with  condition $(58)$ of \cite{cahoy2021flexible} for the alternative generalized Mittag-Leffler distribution \cite{pogany2016probability}. 
\end{rmk}


\section*{Acknowledgments}
The authors G.D.\ and F.P.\ thank the GNAMPA
group of INdAM and acknowledge financial support under the MIUR-PRIN 2022 project \lq\lq Non-Markovian dynamics and non-local equations\rq \rq, no. 202277N5H9. The third author Z.T.\ was partially supported by the project 24-10177L, GACR (LA granty) \lq\lq Fractional and fuzzy-fractional transport in disordered environments\rq \rq.

The authors also thank the Lorentz Center in Leiden, The Netherlands.


\bibliographystyle{abbrv}
\bibliography{refs}

\begin{thebibliography}{10}

\bibitem{abouee2016state}
H.~Abouee-Mehrizi and O.~Baron.
\newblock State-dependent {M}/{G}/1 queueing systems.
\newblock {\em Queueing Systems}, 82(1):121--148, 2016.

\bibitem{abramowitz1968handbook}
M.~Abramowitz and I.~A. Stegun.
\newblock {\em Handbook of mathematical functions with formulas, graphs, and
  mathematical tables}, volume~55.
\newblock US Government printing office, 1968.

\bibitem{bardwell1964two}
G.~E. Bardwell and E.~L. Crow.
\newblock A two-parameter family of hyper-{P}oisson distributions.
\newblock {\em Journal of the American Statistical Association},
  59(305):133--141, 1964.

\bibitem{bartholme2021turan}
C.~Bartholm{\'e} and P.~Patie.
\newblock Turan inequalities and complete monotonicity for a class of entire
  functions.
\newblock {\em Analysis Mathematica}, 47(3):507--527, 2021.

\bibitem{jedidi}
K.~Basalim, S.~Bridaa, and W.~Jedidi.
\newblock Internal {B}ernstein functions and {L}{\'e}vy-{L}aplace exponents.
\newblock {\em Mathematical Methods in the Applied Sciences}, pages 1--11,
  2020.

\bibitem{MR2535014}
L.~Beghin and E.~Orsingher.
\newblock Fractional {P}oisson processes and related planar random motions.
\newblock {\em Electron. J. Probab.}, 14:no. 61, 1790--1827, 2009.

\bibitem{beisemann2022flexible}
M.~Beisemann.
\newblock A flexible approach to modelling over-, under-and equidispersed count
  data in {IRT}: The two-parameter {C}onway--{M}axwell--{P}oisson model.
\newblock {\em British Journal of Mathematical and Statistical Psychology},
  75(3):411--443, 2022.

\bibitem{Bertoin2004SomeCB}
J.~Bertoin, B.~Roynette, and M.~Yor.
\newblock Some connections between (sub)critical branching mechanisms and
  {B}ernstein functions.
\newblock {\em arXiv:math/0412322 [math.PR]}, 2004.

\bibitem{cahoy2021flexible}
D.~O. Cahoy, E.~Di~Nardo, and F.~Polito.
\newblock Flexible models for overdispersed and underdispersed count data.
\newblock {\em Statistical Papers}, 62(6):2969--2990, 2021.

\bibitem{cahoy2013renewal}
D.~O. Cahoy and F.~Polito.
\newblock Renewal processes based on generalized {M}ittag--{L}effler waiting
  times.
\newblock {\em Communications in Nonlinear Science and Numerical Simulation},
  18(3):639--650, 2013.

\bibitem{patie_jacobi}
P.~Cheridito, P.~Patie, A.~Srapionyan, and A.~Vaidyanathan.
\newblock On non-local ergodic {J}acobi semigroups: spectral theory,
  convergence-to-equilibrium and contractivity.
\newblock {\em Journal de l{\textquoteright}\'Ecole polytechnique {\textemdash}
  Math\'ematiques}, 8:331--378, 2021.

\bibitem{consul1973generalization}
P.~C. Consul and G.~C. Jain.
\newblock A generalization of the {P}oisson distribution.
\newblock {\em Technometrics}, 15(4):791--799, 1973.

\bibitem{conway1962queuing}
R.~W. Conway and W.~L. Maxwell.
\newblock A queuing model with state dependent service rates.
\newblock {\em Journal of Industrial Engineering}, 12(2):132--136, 1962.

\bibitem{daly}
F.~Daly and R.~Gaunt.
\newblock The {C}onway-{M}axwell-{P}oisson distribution: Distributional theory
  and approximation.
\newblock {\em ALEA: Latin American Journal of Probability and Mathematical
  Statistics}, 13(2):635–658, 2016.

\bibitem{MR2200069}
J.~del Castillo and M.~P\'{e}rez-Casany.
\newblock Overdispersed and underdispersed {P}oisson generalizations.
\newblock {\em J. Statist. Plann. Inference}, 134(2):486--500, 2005.

\bibitem{dicr2015fractional}
A.~Di~Crescenzo, B.~Martinucci, and A.~Meoli.
\newblock A fractional counting process and its connection with the {P}oisson
  process.
\newblock {\em ALEA, Lat. Am. J. Probab. Math. Stat}, 13:291--307, 2016.

\bibitem{DPS2022}
G.~D’Onofrio, P.~Patie, and L.~Sacerdote.
\newblock Jacobi processes with jumps as neuronal models: A first passage time
  analysis.
\newblock {\em SIAM Journal on Applied Mathematics}, 84(1):189--214, 2024.

\bibitem{garra2018note}
R.~Garra, E.~Orsingher, and F.~Polito.
\newblock A note on {H}adamard fractional differential equations with varying
  coefficients and their applications in probability.
\newblock {\em Mathematics}, 6(1):4, 2018.

\bibitem{MR3171991}
R.~Garra and F.~Polito.
\newblock On some operators involving {H}adamard derivatives.
\newblock {\em Integral Transforms Spec. Funct.}, 24(10):773--782, 2013.

\bibitem{gaunt2019asymptotic}
R.~E. Gaunt, S.~Iyengar, A.~B. Olde~Daalhuis, and B.~Simsek.
\newblock An asymptotic expansion for the normalizing constant of the
  {C}onway--{M}axwell--{P}oisson distribution.
\newblock {\em Annals of the Institute of Statistical Mathematics},
  71:163--180, 2019.

\bibitem{gerhold2012asymptotics}
S.~Gerhold.
\newblock Asymptotics for a variant of the {M}ittag--{L}effler function.
\newblock {\em Integral Transforms and Special Functions}, 23(6):397--403,
  2012.

\bibitem{HWG78}
H.~W. Gould.
\newblock Evaluation of sums of convolved powers using {S}tirling and
  {E}ulerian numbers.
\newblock {\em Fibonacci Quart.}, 16(2):488--497, 1978.

\bibitem{MR3543682}
R.~Herrmann.
\newblock Generalization of the fractional {P}oisson distribution.
\newblock {\em Fract. Calc. Appl. Anal.}, 19(4):832--842, 2016.

\bibitem{hinde1998overdispersion}
J.~Hinde and C.~G. Dem{\'e}trio.
\newblock Overdispersion: models and estimation.
\newblock {\em Computational statistics \& data analysis}, 27(2):151--170,
  1998.

\bibitem{hirsch_yor}
F.~Hirsch and M.~Yor.
\newblock {On the Mellin transforms of the perpetuity and the remainder
  variables associated to a subordinator}.
\newblock {\em Bernoulli}, 19(4):1350 -- 1377, 2013.

\bibitem{hooshmand2022remarks}
M.~Hooshmand.
\newblock Remarks on the functional equation $f (x+ 1)= g (x) f (x)$ and a
  uniqueness theorem for the gamma function.
\newblock {\em Indian Journal of Pure and Applied Mathematics}, pages 1--6,
  2022.

\bibitem{huang2023arbitrarily}
A.~Huang.
\newblock On arbitrarily underdispersed discrete distributions.
\newblock {\em The American Statistician}, 77(1):29--34, 2023.

\bibitem{johnson}
N.~L. Johnson, A.~W. Kemp, and S.~Kotz.
\newblock {\em Univariate discrete distributions}.
\newblock Wiley Series in Probability and Statistics. Wiley-Interscience [John
  Wiley \& Sons], Hoboken, NJ, third edition, 2005.

\bibitem{kaneiwa}
R.~Kaneiwa.
\newblock The formula for higher order derivatives of inverse functions.
\newblock {\em Otaru University of Commerce Humanities Research}, 131:1--3,
  2016.

\bibitem{kemp1968wide}
A.~W. Kemp.
\newblock A wide class of discrete distributions and the associated
  differential equations.
\newblock {\em Sankhy{\=a}: The Indian Journal of Statistics, Series A}, pages
  401--410, 1968.

\bibitem{kemp1974family}
A.~W. Kemp and C.~Kemp.
\newblock A family of discrete distributions defined via their factorial
  moments.
\newblock {\em Communications in Statistics-Theory and Methods},
  3(12):1187--1196, 1974.

\bibitem{MR2388011}
C.~C. Kokonendji, D.~Miz\`ere, and N.~Balakrishnan.
\newblock Connections of the {P}oisson weight function to overdispersion and
  underdispersion.
\newblock {\em J. Statist. Plann. Inference}, 138(5):1287--1296, 2008.

\bibitem{kyprianou2014fluctuations}
A.~Kyprianou.
\newblock {\em Fluctuations of L{\'e}vy Processes with Applications:
  Introductory Lectures}.
\newblock Universitext. Springer Berlin Heidelberg, 2014.

\bibitem{gall1999spatial}
J.~F. Le~Gall.
\newblock {\em Spatial Branching Processes, Random Snakes and Partial
  Differential Equations}.
\newblock Lectures in Mathematics. ETH Z{\"u}rich. Birkh{\"a}user Basel, 1999.

\bibitem{leroy}
E.~Le~Roy.
\newblock {\em Valeurs asymptotiques de certaines s{\'e}ries proc{\'e}dant
  suivant les puissances enti{\`e}res et positives d'une variable r{\'e}elle}.
\newblock Bulletin des sciences math{\'e}matiques, 1900.

\bibitem{macci_pacchiarotti_villa_2022}
C.~Macci, B.~Pacchiarotti, and E.~Villa.
\newblock Asymptotic results for families of random variables having power
  series distributions.
\newblock {\em Modern Stochastics: Theory and Applications}, 9(2):207--228,
  2022.

\bibitem{meerschaert2011fractional}
M.~Meerschaert, E.~Nane, and P.~Vellaisamy.
\newblock The fractional {P}oisson process and the inverse stable subordinator.
\newblock {\em Electron. J. Probab}, 2011.

\bibitem{michelitsch2020generalized}
T.~M. Michelitsch and A.~P. Riascos.
\newblock Generalized fractional {P}oisson process and related stochastic
  dynamics.
\newblock {\em Fractional Calculus and Applied Analysis}, 23(3):656--693, 2020.

\bibitem{minchev2023asymptotics}
M.~Minchev and M.~Savov.
\newblock Asymptotics for densities of exponential functionals of
  subordinators.
\newblock {\em Bernoulli}, 29(4):3307--3333, 2023.

\bibitem{nadarajah2009useful}
S.~Nadarajah.
\newblock Useful moment and {CDF} formulations for the com--poisson
  distribution.
\newblock {\em Statistical Papers}, 50(3):617--622, 2009.

\bibitem{patie2018bernstein}
P.~Patie and M.~Savov.
\newblock Bernstein-gamma functions and exponential functionals of {L}{\'e}vy
  processes.
\newblock {\em Electronic Journal of Probability}, 23:1--101, 2018.

\bibitem{MR4320772}
P.~Patie and M.~Savov.
\newblock Spectral expansions of non-self-adjoint generalized {L}aguerre
  semigroups.
\newblock {\em Mem. Amer. Math. Soc.}, 272(1336):vii+182, 2021.

\bibitem{patie2021self}
P.~Patie and A.~Srapionyan.
\newblock Self-similar {C}auchy problems and generalized {M}ittag-{L}effler
  functions.
\newblock {\em Fractional Calculus and Applied Analysis}, 24(2):447--482, 2021.

\bibitem{pogany2016probability}
T.~K. Pog{\'a}ny and Z.~Tomovski.
\newblock Probability distribution built by {P}rabhakar function. related
  tur{\'a}n and laguerre inequalities.
\newblock {\em Integral Transforms and Special Functions}, 27(10):783--793,
  2016.

\bibitem{puig2023some}
P.~Puig, J.~Valero, and A.~Fern{\'a}ndez-Fontelo.
\newblock Some mechanisms leading to underdispersion: Old and new proposals.
\newblock {\em Scandinavian Journal of Statistics}, 2023.

\bibitem{pujol2014new}
M.~Pujol, J.-F. Barquinero, P.~Puig, R.~Puig, M.~R. Caball{\'\i}n, and
  L.~Barrios.
\newblock A new model of biodosimetry to integrate low and high doses.
\newblock {\em PLoS One}, 9(12):e114137, 2014.

\bibitem{schilling2012bernstein}
R.~L. Schilling, R.~Song, and Z.~Vondracek.
\newblock {\em Bernstein functions}.
\newblock de Gruyter, 2012.

\bibitem{sellers2023conway}
K.~F. Sellers.
\newblock {\em The {C}onway--{M}axwell--{P}oisson distribution}, volume~8.
\newblock Cambridge University Press, 2023.

\bibitem{sellers2012poisson}
K.~F. Sellers, S.~Borle, and G.~Shmueli.
\newblock The {COM}-poisson model for count data: a survey of methods and
  applications.
\newblock {\em Applied Stochastic Models in Business and Industry},
  28(2):104--116, 2012.

\bibitem{sellers2010flexible}
K.~F. Sellers and G.~Shmueli.
\newblock A flexible regression model for count data.
\newblock {\em The Annals of Applied Statistics}, pages 943--961, 2010.

\bibitem{shmueli2005useful}
G.~Shmueli, T.~P. Minka, J.~B. Kadane, S.~Borle, and P.~Boatwright.
\newblock A useful distribution for fitting discrete data: revival of the
  {C}onway--{M}axwell--{P}oisson distribution.
\newblock {\em Journal of the Royal Statistical Society Series C: Applied
  Statistics}, 54(1):127--142, 2005.

\bibitem{Shmueli2005}
G.~Shmueli, T.~P. Minka, J.~B. Kadane, S.~Borle, and P.~Boatwright.
\newblock A useful distribution for fitting discrete data: revival of the
  {C}onway--{M}axwell–{P}oisson distribution.
\newblock {\em Journal of the Royal Statistical Society: Series C (Applied
  Statistics)}, 54(1):127--142, 2005.

\bibitem{simon2022remark}
T.~Simon.
\newblock Remark on a {M}ittag--{L}effler function of {L}e {R}oy type.
\newblock {\em Integral Transforms and Special Functions}, 33(2):108--114,
  2022.

\bibitem{siri2012parametric}
P.~Siri, E.~Henninger, and M.~P. Sormani.
\newblock A parametric model fitting time to first event for overdispersed
  data: application to time to relapse in multiple sclerosis.
\newblock {\em Lifetime data analysis}, 18:139--156, 2012.

\bibitem{tomovskipoganj2011}
Z.~Tomovski and T.~Pog{\'a}nj.
\newblock Integral expressions for {M}athieu–type power series and for the
  {B}utzer--{F}locke--{H}auss ${\Omega}$ function.
\newblock {\em Fract. Calc. Appl. Anal.}, 14(4):623--634, 2011.

\bibitem{tomovski2014laplace}
Z.~Tomovski, T.~K. Pog{\'a}ny, and H.~M. Srivastava.
\newblock Laplace type integral expressions for a certain three-parameter
  family of generalized {M}ittag--{L}effler functions with applications
  involving complete monotonicity.
\newblock {\em Journal of the Franklin Institute}, 351(12):5437--5454, 2014.

\bibitem{watson1966handbook}
G.~N. Watson.
\newblock {\em A Treatise on the Theory of {B}essel Functions, 2nd edition}.
\newblock Cambridge University Press, Cambridge, 1966.

\bibitem{webster}
R.~Webster.
\newblock Log-convex solutions to the functional equation $f (x+ 1)= g (x) f
  (x)$: $\gamma$-type functions.
\newblock {\em Journal of Mathematical Analysis and Applications},
  209(2):605--623, 1997.

\bibitem{wu}
G.~Wu, S.~H. Holan, C.~H. Nilon, and C.~K. Wikle.
\newblock {Bayesian binomial mixture models for estimating abundance in
  ecological monitoring studies}.
\newblock {\em The Annals of Applied Statistics}, 9(1):1 -- 26, 2015.

\bibitem{zhu2012modeling}
F.~Zhu.
\newblock Modeling overdispersed or underdispersed count data with generalized
  {P}oisson integer-valued {GARCH} models.
\newblock {\em Journal of Mathematical Analysis and Applications},
  389(1):58--71, 2012.

\end{thebibliography}

\end{document}